\documentclass[a4paper,reqno,oneside,11pt]{amsart}

\usepackage{fullpage}

\usepackage{enumerate}
\usepackage{amsmath}
\usepackage{amsfonts}
\usepackage{amssymb}
\usepackage{verbatim}
\usepackage[colorlinks]{hyperref}

\newtheorem{lemma}{Lemma}

\newtheorem{claim}[lemma]{Claim}
\newtheorem{theorem}[lemma]{Theorem}
\newtheorem{conjecture}[lemma]{Conjecture}

\newtheorem{definition}[lemma]{Definition}

\newtheorem{proposition}[lemma]{Proposition}

\newcommand{\Length}{\mathrm{Length}}
\newcommand{\modd}{\mathrm{mod}}

\author{	J\'ozsef Balogh} \thanks{Department of Mathematical Sciences, University of Illinois at Urbana-Champaign, Urbana, Illinois 61801, USA, 
and Moscow Institute of Physics and Technology,  Russian Federation. Partially supported by NSF
Grants DMS-1500121 and DMS-1764123, Arnold O. Beckman Research Award (UIUC Campus Research Board RB 18132) and
the Langan Scholar Fund (UIUC)}
	\author {Gal Kronenberg}
	\thanks{School of Mathematical Sciences, Raymond and Beverly Sackler Faculty of Exact Sciences, Tel Aviv University, Tel Aviv, 6997801, Israel. Email: galkrone@mail.tau.ac.il}
	\author {Alexey Pokrovskiy} \thanks{Department of Economics, Mathematics, and Statistics, Birkbeck,
University of London, United Kingdom. Email:
Dr.Alexey.Pokrovskiy@gmail.com.}
	\author {Tibor Szab\'o} \thanks{Institute of Mathematics, FU Berlin, 14195 Berlin; e-mail: szabo@math.fu-berlin.de. Research supported in part by GIF grant No. G-1347-304.6/2016.}
	
\date{\today}

\title{The maximum length of $K_r$-Bootstrap Percolation}

\begin{document}
	\maketitle

\begin{abstract}
	  \textit{Graph-bootstrap percolation}, also known as weak saturation, was introduced by Bollob\'as in 1968. In this process, we start with initial ``infected" set of edges $E_0$, and we infect new edges according to a predetermined rule. Given a graph $H$ and a set of previously infected edges $E_t\subseteq E(K_n)$, we infect a non-infected edge $e$ if it completes a new copy of $H$ in $G=([n],E_t\cup e)$. A question raised by Bollob\'as asks for the maximum time the process can run before it stabilizes.  Bollob\'as,  Przykucki,  Riordan, and  Sahasrabudhe considered this problem for the most natural case where $H=K_r$. They answered the question for $r\leq 4$ and gave a non-trivial  lower bound for every $r\geq 5$. They also conjectured that the maximal running time is $o(n^2)$ for every integer $r$. 
	  In this paper we disprove their conjecture for every $r\geq 6$ and we give a better lower bound for the case  $r=5$; in the proof   we use the Behrend construction.
\end{abstract}	 
	
	\section{Introduction}
	
	Weak saturation of graphs was introduced by Bollob\'as~\cite{weakSat} in 1968. A graph $G$ on $n$ vertices is \textit{weakly saturated} with respect to a graph $H$, if $G$ has no copies of $H$, but there exists an ordering of $E(K_n)\setminus E(G)=\{e_1,\dots,e_t\}$ such that the addition of $e_i$ to $G\cup \{e_1,\dots,e_{i-1}\}$ will create a new copy of $H$, for every $i\in [t]$. It was later noticed by  Balogh, Bollob\'as and Morris~\cite{GraphBootstrapPer} that weak saturation is strongly related to the so-called \textit{bootstrap percolation process}, which is a type of cellular automata  introduced in 1979 by Chalupa, Leath, and Reich~\cite{Chalupa1979}, see also~\cite{aizleb,grid}.

	For our setting, we first redefine the notion of a weakly saturated graph in terms
	of an infection process, known as the \textit{graph-bootstrap percolation}. 
	For graphs $F,H$ we describe the \textit{$(F,H)$-bootstrap process} as follows. We start with an initial \textit{infected set of edges} $E_0\subseteq E(F)$ and write $G_0 :=(V(F),E_0)$ (sometimes called the \textit{starting graph}). At each step, an edge of $F$ becomes infected if it completes an infected copy of $H$. More formally, denote by $n_H(G)$ the number of copies of $H$ in a graph $G$. Let	
	$$G_t=G_{t-1}\cup\{ e\in E(F) \mid n_H(G_{t-1}\cup \{e\}) >n_H(G_{t-1}) \}\quad  \text{  and } \quad E_t=E(G_t).$$
	 We say that the \textit{running time} of the $(F,H)$-bootstrap process is $t$, if $t$ is the minimum integer  such that $G_{t+1}=G_t$. In this case we say that $\langle G_0 \rangle_{(F,H)}:=G_t$ is the \textit{final graph}
	 and the process {\it stabilizes} in $t$ steps. We say
	that $E_0$ \textit{percolates} if every edge of $F$ is eventually infected, that is, if  the final graph $\langle G_0 \rangle_{(F,H)}=F$.
	In the special case when $F=K_n$, we refer to the $(K_n,H)$-bootstrap percolation process as the $H$-process. 

The origins of the concept in physics involved investigating the threshold phenomena of percolation in various random setups~\cite{aizleb,grid}. Balogh,  Bollob\'as,  and  Morris~\cite{GraphBootstrapPer} studied the threshold probability $p_c(n,H)$ for the event that  the Erd\H os-R\'enyi random graph $G(n,p)$  percolates in the $H$-bootstrap percolation process.
 It is easy to see that a starting graph $G_0$ on $n$ vertices percolates in the $K_3$-process if and only if $G_0$ is connected, so $p_c(n,K_3)=\frac{\log{n}}{n}+\Theta\left(\frac{1}{n}\right)$
by the result of Erd\H{o}s and R\'enyi~\cite{ErdosRenyi} about the threshold for connectivity of $G(n,p)$.
Determining the value of $p_c(n,K_r)$ is much more difficult when $r\geq 4$. It was shown in~\cite{GraphBootstrapPer} that $p_c(n,K_4)=\Theta\left(\sqrt{\frac{1}{n\log n}}\right)$, and that for $r\geq 5$, $ \frac {n^{-1/\lambda(r)}}{2e\log n}\leq  p_c(n,K_r) \leq n^{-1/\lambda(r)}\log n $, where $\lambda(r)=\frac{\binom{r}{2}-2}{r-2}$.
	The bounds on $p_c(n,K_4)$ was later improved in~\cite{AngelKolesnikK4,KolesnikSharpK4}. Recently, this problem was also studied for more general graphs $H$ (see also~\cite{MorrisionNoelHypergraph}).

Much work has been done on the extremal properties of the $K_r$-process. 
Alon~\cite{AlonWeakSat1985}, Frankl~\cite{FranklWeakSat1982}, and Kalai~\cite{KalaiWeakSat1984}, showed that the smallest percolating set of edges in the $K_r$-process in $K_n$ has size $\binom{n}{2}-\binom{n-r+2}{2}$, realized for example by $K_n-K_{n-r+2}$,  thus confirming a conjecture of Bollob\'as~\cite{weakSat}. 
This question was also studied for other graphs $F$ and $H$, see~\cite{AlonWeakSat1985,LinearAlgebraBootstrapPer,BaloghPete1998,MorrisionNoelHypercube}.
	
Despite missing almost all the ${n\choose 2}$ edges of $K_n$, the $K_r$-percolating starting graph $K_n-K_{n-r+2}$ percolates very fast: every non-edge is the sole missing edge from a copy of $K_r$, so is added simultanously in the very first step of the process. 
Nevertheless, in some of the  applications the speed of percolation is quite relevant.  
In this direction Bollob\'as raised the extremal problem of determining the \textbf{slowest} percolating set in the bootstrap process (i.e.~the one that  has the maximum running time).  Benevides and Przykucki~\cite{BenevidesPrzykucki1,BenevidesPrzykucki2,Przykucki0} studied this problem in the related setting of neighborhood percolation. Gunderson, Koch and Przykucki \cite{TimeOfBootstrapRandom} considered a `percolation
time' problem in the random setting.
The question for the $K_r$-process on the edges was investigated independently by Bollob\'as, Przykucki, Riordan and Sahasrabudhe~\cite{MaxRunTime} and Matzke~\cite{Matzke}. They defined
$$M_r(n)=\max \{t \mid \text{$\exists G_0\subseteq K_n$ such that $G_t\neq G_{t-1}$ in the $K_r$-bootstrap process} \}$$
to be the maximum running time for the $K_r$-bootstrap percolation on $n$ vertices until it stabilizes, taken over all starting graphs. 
It is easy to see that in the $K_3$-process the diameter of the infected graph decreases at least by a factor of  two in each step, and hence 
$M_3(n) = \lceil \log_2 (n-1)\rceil$. For the $K_4$-process the precise answer was found and turned out to be linear in $n$.
 
	\begin{theorem}[\cite{MaxRunTime,Matzke}]
		$M_4(n)=n-3$ for all $n\geq 3$.
	\end{theorem}
In ~\cite{MaxRunTime} for $r\geq 5$  
subquadratic polynomial lower bounds with the exponents tending to $2$ 
as $r$ tends to infinity were given (see also \cite{Matzke} for $r=5$).
\begin{theorem}[Theorem 2 in~\cite{MaxRunTime}]\label{thm:MaxRunTime}
	For each fixed $r\geq 5$, we have $M_r(n)\geq n^{2-\frac{1}{\lambda (r)}-o(1)}$  as $n\to \infty$, where $\lambda (r)=\frac{\binom r2-2}{r-2}$.
\end{theorem}
	
Note the reappearance of the constant $\lambda(r)$ from the bounds on $p_c(n,K_r)$ mentioned above (see also \cite{GraphBootstrapPer}). Consequently the lower bound in Theorem~\ref{thm:MaxRunTime} is around the number when a typical set of that many edges starts to percolate. This is so, as in \cite{MaxRunTime} a random construction (though different from $G(n,p_c(n,K_r))$) is used to obtain the lower bound on $M_r(n)$.  
As for an upper bound, in \cite{MaxRunTime} it was conjectured that the running time of any $K_r$-bootstrap percolation process should be subquadratic for $r\geq 5$.
	\begin{conjecture}\label{conj:main}\cite{MaxRunTime} 
		 For all $r \geq 5$ we have $M_r(n) = o(n^2)$.
	\end{conjecture}

In our first main result, we disprove Conjecture~\ref{conj:main} for all $r\geq 6$.

	\begin{theorem}
		\label{thm:mainKr}
		For every $r\geq 6$ and large enough $n$, we have $M_r(n) \geq \frac{n^2}{2500}$.
	\end{theorem}

Our construction of the starting graph for the slow $K_r$-process does not obviously extend to $r=5$. Nevertheless, some of the ideas can be salvaged by utilizing sets of integers having no arithmetic progression of length three. Using the relevant constructions from additive number theory 
allows us to improve the lower bound of ~\cite{MaxRunTime} for $M_5(n)$ to almost quadratic.

A set $B$ of numbers is called {\em $3$-AP-free} if for any $b_1, b_2, b_3 \in B$ with $2b_1 = b_2+b_3$, we have $2b_1 = b_2+b_3$. Denote  $r_3(n)$  the largest cardinality of a $3$-AP-free subset of $[n]$. Determining $r_3(n)$ (and its generalization for $k$-AP-free subsets) is a cornerstone problem  in  additive number theory, with a rich history that also involves inspiring a significant portion of modern combinatorics. Behrend~\cite{Behrend} showed that there are
$3$-AP-free subsets of $[n]$ of size $n^{1-1/O(\sqrt{\log n})}$. From the other side
the function $r_3(n)$ was shown to be $o(n)$ by Roth in 1953 using analytic number theory. 
Later this was also proved by various other methods, including combinatorics, ergodic theory, and non-standard analysis.
 
Here we connect $r_3(n)$ to the extremal function $M_5(n)$ of slow $K_5$-bootstrap percolation.
	\begin{theorem}
	\label{thm:mainK5}  
$M_5(n) \geq \frac{nr_3(n)}{1200}$. 
In particular, $M_5(n)\geq n^{2-O(1/\sqrt{\log n})}$.
\end{theorem}

The above theorem gives a significant improvement on the previously best lower bound $M_5(n) \geq n^{13/8-o(1)}$ from Theorem~\ref{thm:MaxRunTime}.
Conjecture~\ref{conj:main} is still open for $r=5$ and we tend to agree with the authors of \cite{MaxRunTime} about its validity. The positive resolution of Conjecture~\ref{conj:main} would closely tie the classic additive number theoretic function $r_3(n)$ to percolation. 
 
In the next section we start with some  basic terminology, and present simple number-theoretic facts that are used in the following sections.  Additionally, we prove a general lower bound for the maximum length of the $K_r$-bootstrap percolation process based on an $r$-uniform hypergraph with certain properties. In Section~\ref{sec:Kr} we prove Theorem~\ref{thm:mainKr}, and in Section~\ref{sec:K5} we prove Theorem~\ref{thm:mainK5}.

	\bigskip 
	
	\section{Preliminaries}
	Our graph-theoretic notation is standard, 
	in particular we use the following. For a graph $G = (V,E)$ and a set $U\subset V$, let $G[U]$ denote the corresponding vertex-induced subgraph of $G$. We also denote $e(G)=|E(G)|$. For a (hyper)graph $\mathcal H$ and $e\in E(\mathcal H)$, we denote by $\mathcal H-e$ the spanning (hyper)graph obtained from $\mathcal H$ after removing only the edge $e$. For $e\subseteq V(\mathcal H)$ where $e\notin E(\mathcal H)$ we denote by $\mathcal H\cup e$ the (hyper)graph obtained by adding the edge $e$ to the edges of  $\mathcal H$.
	For two hypergraphs $\mathcal H$ and $\mathcal H'$, we let $\mathcal H\cup \mathcal H'$ be the hypergraph with $V(\mathcal H\cup \mathcal H')=V(\mathcal H)\cup V(\mathcal H')$ and $E(\mathcal H\cup \mathcal H')=E(\mathcal H)\cup E(\mathcal H')$. 
The edges of an $r$-uniform hypergraphs are referred  as $r$-edges, sometimes including $r=2$. For a set of vertices U, we denote by $\binom{U}{2}$ the set of all pairs of vertices from $U$.	
Whenever the  reference to the vertex set is not crucial we tend to identify a (hyper)graph $\mathcal H$ with its (hyper)edge set and write $\mathcal H$ instead of $E(\mathcal H)$.

In the proofs of the main theorems, we need to construct (simple) graphs on which the running time of the  graph-bootstrap percolation will be long enough. In order to do so, we first construct auxiliary hypergraphs, and then  consider  the graphs obtained from them. Generally speaking, we want to find $r$-uniform hyergraphs for which each $r$-edge will represent a potential copy of $K_r$, and the intersection between the $r$-edges that has size 2 will represent exactly the 2-edges that we add during the percolation.    Denote  $K_r^-$  the complete graph on $r$ vertices with one edge deleted.
	
		\begin{definition}[2-skeleton]
		The {\em 2-skeleton of a hypergraph $\mathcal H$} is the graph $G = G({\mathcal H})$ with $V(G)=V(\mathcal H)$, and $E(G)=\{ab: \{a,b\}\subseteq e \text{ for some } e\in E(\mathcal H)\}$.
	\end{definition}

\begin{definition}[Induced $H$-free]\label{Definition_InducedKrFree}
	An $r$-uniform hypergraph $\mathcal H$ is {\em induced $H$-free} if every copy of $H$ in the $2$-skeleton of $\mathcal H$ is contained in an edge of $\mathcal H$.
\end{definition}

	\begin{lemma}[Key Lemma]\label{Lem:KeyLemma}
If there exists an $r$-uniform hypergraph $\mathcal H$ on $n$ vertices with an ordering $e_1, e_2, \dots, e_{m} \in \mathcal H$ of its edges, such that 
\begin{itemize}
\item[(i)] $\mathcal H$ is induced $K_r^-$-free,
\item[(ii)] there exist subsets $f_i  \subseteq e_i$ of size $|f_i|=2$ for $1 \leq i \leq m$, such that $f_i \subseteq e_j$ if and only if $i=m=j$ or $i <m$ and $j = i, i+1$,
\end{itemize}
then the $K_r$-process starting with the subgraph 
$G_0 : = G - \{f_i : i=1, \ldots, m \}$  of the $2$-skeleton $G$ of 
$\mathcal H$ has length at least $m$. 
In particular,
$M_r(n) \geq m  =e(\mathcal H)$.
\end{lemma}
\begin{proof} 
We prove by induction that for every $1 \leq i\leq m$, before the $i$th step of the $K_r$-process, $G_{i-1} = G -  \{ f_i, \ldots , f_m\}$.
From this the lemma follows immediately: $e(G_i) \setminus e(G_{i-1}) = \{ f_i\} \neq \emptyset$  for every $i =1, \ldots , m$, since by (ii) $f_i=f_j$ if and only if $i=j$. So  the $K_r$-process stabilizes only after at least $m$ steps, indeed. 

To start the induction, $G_0 = G -\{ f_1, \ldots , f_m\}$ is true by definition. 
Let us assume now that $G_{i-1} = G- \{f_i, \ldots , f_m\}$ for some $i>1$. 
By condition (i) and Definition~\ref{Definition_InducedKrFree},  every $K_r^-$ in $G$, and hence also in any of its subgraphs, like $G_{i-1}$, is contained in some $e_j$, $j=1, \ldots , m$. So in the $i$th step of the $K_r$-process a new $K_r$ can be created only from   these $K_r^-$'s. Since $G[e_j] \simeq K_r$ for every $j$ and
$G_{i-1} = G- \{f_i, \ldots , f_m\}$, condition (ii) implies that  $G_{i-1}[e_j]$ is
\begin{itemize}
\item a $K_r$ for every $j < i$, 
\item a $K_r$ minus the edge $f_i$ for $j=i$, and 
\item a $K_r$ minus the two edges $f_{j-1}$ and $f_j$ for $j>i$.
\end{itemize}
Consequently, in the $i$-th step of the $K_r$-process exactly  one new $K_r$ is created: the one the edge $f_i$ completes on the set $e_i$. 
This shows that $G_i = G_{i-1} \cup f_i = G- \{ f_{i+1}, \ldots , f_m\}$. 
\end{proof}

	The following technical lemma will be useful in the proof of Theorem~\ref{thm:mainKr}. 
	\begin{lemma}\label{NumberTheoryLemma}
		Let $n\geq 10$ 
and $\ell=n+20$ be integers, and let $d,s_1, s_2$ be integers such that 
$|d| \leq n^2/100$ and $|s_1|, |s_2|\leq 2$. 
If $d \equiv s_1 \ (\modd\ n) $ and $d\equiv s_2\  (\modd\ \ell)$ 
then $d =s_1=s_2$. 
	\end{lemma}
	
	\begin{proof}
		There exist integers $k_1, k_2$ such that $d=k_1n+s_1$ and 
$d=k_2\ell +s_2=k_2n+20k_2+s_2$. %, where $0\leq s_1,s_2\leq s$.
 Subtracting these from each other gives $(k_1-k_2)n=20k_2+ s_1-s_2$.

If $k_1\neq k_2$, then the absolute value of the left hand side 
is at least $n$.
For the absolute value of the right hand side we have 
$|20k_2 +s_1-s_2|= \left| 20\frac{d-s_2}{n+20} +s_1 - s_2\right| \leq 
20\frac{|d|}{n} + |s_1|+ |s_2| \leq  n/5 + 4$, a contradiction to the lower bound on $n$. 

If $k_1=k_2\neq 0$, then we have $20 \leq |-20k_2| = |s_1-s_2| \leq 4$ which is a contradiction. Hence,  $k_1=k_2 = 0$, in which case $0=s_1-s_2$ and $d=s_1=s_2$.
	\end{proof}

	\bigskip
	
	\section{Proof of Theorem~\ref{thm:mainKr} ($K_r$-Bootstrap Percolation for $r\geq 6$).}\label{sec:Kr}

We start by observing that the theorem can easily be reduced to the case of $r=6$.

\begin{proposition} \label{prop:reduction} $M_r(n) \leq M_{r+1}(n+1)$.
\end{proposition}

\begin{proof}
For a graph $\Gamma $ and a vertex $v\notin V(\Gamma)$, denote by $\Gamma\vee v$ the graph obtained by adding $v$ to $\Gamma$ and all the edges $\{vu \mid u\in V(\Gamma)\}$. 
Observe that for any $\Gamma \subseteq K_n$ the set
$$\{ e \in E(K_n) \setminus E(\Gamma) : \mbox{$\exists L\subseteq V(K_n)$ with $\Gamma[L\cup e] \simeq K_r^-$} \}$$ 
of edges added to $\Gamma$ in the first step of the $K_r$-process is the same as the set  $$\{ e \in E(K_n\vee v) \setminus E(\Gamma \vee v) : \mbox{$\exists L\subseteq V(K_n)$ with $\Gamma[L\cup \{ v\} \cup e] \simeq K_{r+1}^-$} \}  $$   
of edges added to $\Gamma \vee v$ in the $K_{r+1}$-process. Then
the proposition follows immediately since then for any starting graph  $G_0 \subseteq K_n$ the $K_r$-process adds edges in the exact same order as does the $K_{r+1}$-process with starting graph $G_0\vee v \subseteq K_n \vee v$ and hence it also lasts exactly as long.
\end{proof}

In the rest of this section we show that Theorem~\ref{thm:mainKr} holds for $r=6$.
\begin{lemma} \label{lemma:r=6}
$M_6(n) \geq \frac{n^2}{2000}.$
\end{lemma}
From this lemma and the above proposition our theorem follows easily.
\begin{proof}[Proof of Theorem~\ref{thm:mainKr}]
By repeated application of Proposition~\ref{prop:reduction} and then of Lemma~\ref{lemma:r=6}
we have that for every fixed $r\geq 6$ and every sufficiently large  $n$,
$$M_r(n)\geq M_6(n-r+6)\geq \frac 1{2000}(n-r+6)^2\geq \frac 1{2500}n^2.$$
\end{proof}

In order to prove \ref{lemma:r=6} we define a $6$-uniform 
hypergraph ${\mathcal H}(n)$ on $\Theta (n)$ vertices with $\Theta (n^2)$ edges which satisfies the conditions  of Lemma~\ref{Lem:KeyLemma}.

	\begin{definition}[$\mathcal H(n)$]\label{Def:KrHypergraph}
		Let  $n\ge 10,$ and  $\ell=n+20$ be positive integers. The $6$-uniform hypergraph ${\mathcal H}= {\mathcal H}(n)$ is defined on the vertex set 
$$V(\mathcal H)= X\cup Z \cup Y \cup W$$
where $X=\{x_0,x_1,\dots,x_{n-1}\}$, $Z=\{z_0,z_1,\dots,z_{\ell-1}\}$, $Y=\{y_0,y_1,\dots,y_{n-1}\}$,  $W=\{w_0,w_1,\dots,w_{\ell-1}\}$, 
are four pairwise disjoint sets. 
Let $\mathcal H=\{ e_t \mid  0\leq t\leq  m -1 \},$ where $m =\lfloor {n^2}/{100}\rfloor$ and for every $0\leq t \leq m-1$,  we denote
$$e_t=\{x_{t (\modd\ n)},\  x_{t+1 (\modd\ n)},\ y_{t+1 (\modd\ n)},  \  z_{t (\modd\ \ell)},\  z_{t+1 (\modd\ \ell)},\ w_{t+1 (\modd\ \ell)}\}.$$ 
	\end{definition}

First, we count the vertices and edges of $\mathcal H(n)$ and define the appropriate pairs $f_i \subseteq e_i$ for the use of the Key Lemma.

\begin{proposition}\label{prop:simple-properties} For $0\leq t \leq m-1$, define $f_{t} := \{ x_{t+1 (\modd\ n)}, z_{t+1 (\modd\ \ell)}\}$. 
\begin{itemize}  
	\item[(a)] $\mathcal H(n)$ has  $4n +40$ vertices.
\item[(b)] If $0\leq t\neq j\leq m-1$,  then $e_t\neq e_{j}$. In particular, 
$\mathcal H(n)$ has $m$ different edges.
\item[(c)] $f_{t} \subseteq e_j$ if and only if $t <m-1$ and $j = t, t+1$ or $t=m-1=j$.
\end{itemize}
\end{proposition}

\begin{proof} Part (a) follows by adding up the sizes of participating pairwise disjoint sets.
For part (b) note that if $e_t=e_j$ then $ t \equiv j\ (\modd\ n)$ and $t\equiv j\ (\modd\ \ell)$ so by Lemma~\ref{NumberTheoryLemma} we have $t=j$. 
The ``if'' direction of part (c) can be read off from the definitions of $f_t$ and $e_j$. For the other direction suppose that $f_{t}=\{x_{t+1 (\modd\ n)},z_{t +1(\modd\ \ell)}\}\subseteq e_j$. Then it follows that $x_{t+1 (\modd\ n)}=x_{j (\modd\ n)}$ or $x_{j+1 (\modd\ n)}$, and $z_{t+1 (\modd\ \ell)}=z_{j (\modd\ \ell)}$ or $z_{j+1 (\modd\ \ell)}$, which means that $j-t \equiv 0$ or $1\ (\modd\ n)$, and $j-t  \equiv 0$ or $1\ (\modd\ \ell)$. By Lemma~\ref{NumberTheoryLemma} then we have $j-t = 0$ or $1$. 
	\end{proof}
	
In part (c) of  Proposition~\ref{prop:simple-properties} we have verified condition (ii) of Lemma~\ref{Lem:KeyLemma} for $\mathcal H(n)$. In the rest of this section we verify condition (i), that is, we show that $\mathcal H(n)$ is induced $K_6^-$-free. 
Let $G = G(n)$ denote the $2$-skeleton of $\mathcal H(n)$. 

\begin{claim}[Cliques on a side of  the 2-skeleton]\label{clique2}
Let $U\subseteq X\cup Y$ (or $U\subseteq Z \cup W$) be a set of vertices such that $G[U]$ is a clique. % that induces a clique in $G$. 
Then $|U|\leq 3$ and $U\subseteq e_t$ for some $0\leq t \leq n-1$ (or some $0\leq t \leq \ell -1$).
\end{claim}
\begin{proof} 
Let us assume that $U\subseteq X\cup Y$, the proof of the case when $U\subseteq Z\cup W$ is analogous. 
By the definition of the edges of $\mathcal H(n)$ the restriction of the $2$-skeleton $G$ to $X\cup Y$ is the union of $n$ edge disjoint  triangles, one for each $t=0, \ldots , n-1$ on the vertex set $\{x_t, x_{t+1 (\modd\ n)}, y_{t+1 (\modd\ n)}\} $. 
In particular,  if $y_{t+1 (\modd\ n)} \in U$ 
 then $U \subseteq e_{t}$. Otherwise $U\cap (X\cup Y)\subseteq X$ and has size at most two, since $G[X]$ is a cycle of length $n \geq 4$. We then  conclude that $U= \{x_t, x_{t+1 (\modd\ n)}\}\subset e_t$, for some $t = 0, \ldots , n-1$.
\end{proof}

\begin{claim}[Copies of $K_{5}$ in $G$]\label{cl:Kr-1InEdge}
If $G[U] \simeq K_{5}$ then there exists a $t$ such that 
$U \subset e_t$. 
\end{claim}	

\begin{proof}
	Since $U$ spans a clique in $G$,  the graphs induced by both $U_1=U\cap(X\cup Y)$ and $U_2=U\cap(Z\cup W)$ are also cliques.
By Claim~\ref{clique2}, $U_1 \subseteq e_q $ for some $0\leq q \leq n-1$, and $U_2 \subseteq e_s$ for some $0\leq s \leq \ell-1$, hence 
$$ U_1 \cup U_2 \subseteq \{ x_q,\ x_{q+1 (\modd\ n)},\ y_{q+1 (\modd\ n)},\  z_s,\  z_{s+1 (\modd\ \ell)},\ w_{s+1 (\modd\ \ell)} \}.$$
We show that $e_t = \{ x_q, x_{q+1 (\modd\ n)},y_{q+1 (\modd\ n)}, z_s, z_{s+1 (\modd\ \ell)},w_{s+1 (\modd\ \ell)} \}$ for some  $e_t \in {\mathcal H}(n)$, which happens if there is an integer $t\le m$, such that $q \equiv t (\modd\ n)$ and $s\equiv t (\modd\ \ell)$.
This is certainly the case if both $y_{q+1 (\modd\ n)}$ and $w_{s+1 (\modd\ \ell)}$ are in $U$, since then they are adjacent in the $2$-skeleton $G$ and hence there exists a $t$ such that $e_t$ contains both, implying the required congruences. 
Otherwise, exactly one of $y_{q+1 (\modd\ n)}$ and $w_{s+1 (\modd\ \ell)}$ is in $U$, say $w_{s+1 (\modd\ \ell)} \in U_2$ and is adjacent to both $x_q$ and $x_{q+1 (\modd\ n)}$ (which form $U_1$). Then there exist  $t$ and $t'$ such that $\{w_{s+1 (\modd\ \ell)},x_q\}\subseteq e_t$ and $\{w_{s+1 (\modd\ \ell)},x_{q+1 (mod\ n)}\}\subseteq e_{t'}$. 
This implies $q\equiv t$ or $t+1$, and $q+1 \equiv t'$ or $t'+1$ $(\modd\ n)$. 
In any case  $1 = (q+1)-q \equiv t' -t + \{ 1, 0, -1\} \ (\modd \ n)$, so $t'-t \equiv \{ 0,1,2\}\  (\modd\ n)$.  
Furthermore $s+1 \equiv t+1$ and $s+1 \equiv t'+1 \ (\modd\ \ell)$, so  $t'-t \equiv 0 \ (\modd \ \ell)$.
By Lemma~\ref{NumberTheoryLemma} we get $t'-t = 0$ and hence the required congruences $q \equiv t (\modd\ n)$ and $s\equiv t (\modd\ \ell)$ hold.
\end{proof}

	\begin{claim}[Induced $K_6^-$-freeness]\label{cor:InducedKrfree}
	The hypergraph	$\mathcal H_6(n)$ is induced $K_6^-$-free.
	\end{claim}
	
	\begin{proof}  Let $K'$ be a copy of $K_6^-$ in $G$. Since $K_6^-$ consists of two copies of $K_{5}$ intersecting in $4$ vertices and by Claim~\ref{cl:Kr-1InEdge} each of these copies is contained in a hyperedge of ${\mathcal H(n)}$, there exist $0\leq t, t' < m$ such that $V(K') \subseteq e_t\cup e_{t'}$ and $|e_t\cap e_{t'}|\geq 4$. 
We show now that if $t\neq t'$ then $|e_t\cap e_{t'}| \leq 3$, implying that $t=t'$ and hence that $V(K') \subseteq e_t$, as required.

Observe
that $|e_t \cap e_{t'} \cap (X\cup Y)|$ is equal to 
$3$ if $ t \equiv t' (\modd\ n)$, equal to $1$ if $t \equiv t'-1$ or $t'+1 (\modd\ n)$, and equal to $0$ otherwise. 
Analogously, $|e_t \cap e_{t'} \cap (Z\cup W)|$ is equal to 
$3$ if $ t \equiv t' (\modd\ \ell)$, equal to $1$ if $t \equiv t'-1$ or $t'+1 (\modd\ \ell)$, and equal to $0$ otherwise. 

Consequently $6 > |e_t\cap e_{t'}| = |e_t \cap e_{t'} \cap (X\cup Y)| + |e_t \cap e_{t'} \cap (Z\cup W)| \geq 4$ would require that  
$ t - t' \equiv 0 \ (\modd\ n)$ and $t - t' \equiv \{1, -1\} \ (\modd\ \ell)$ (or the same congruences with $n$ and $\ell$ switched). This is impossible (in either case) by Lemma~\ref{NumberTheoryLemma}, since $ 0 \not\in \{ 1, -1\}$.
	\end{proof}

\begin{proof}[Proof of Lemma~\ref{lemma:r=6}] 
Above we have checked that the hypergraph $\mathcal H(n)$ satisfies both conditions of Lemma~\ref{Lem:KeyLemma}, hence $M_6(|V({\mathcal H(n)})|) \geq |E({\mathcal H(n)})|$. Then by parts (a) and (b) of Proposition~\ref{prop:simple-properties} we have that for every 
$N$ sufficiently large,
$$M_6(N) \geq M_6(4n+40) = M_6(|V({\mathcal H(n)})|) \geq \left\lfloor \frac{n^2}{100} \right\rfloor \geq \frac{1}{2000}N^2,$$
where $n$ is the unique integer such that $4n+40 < N \leq 4n+43$.
\end{proof}

	\section{Proof of Theorem~\ref{thm:mainK5} ($K_5$-Bootstrap Percolation).}\label{sec:K5}
	
The construction we introduced for $r=6$  fails to extend for the case that $r=5$.
While for $r=6$, we were able to show that the only copies of $K_{r-1}$ are inside the edges of the hypergraph 
(and thus no extra copies of $K_r^-$ can appear), this is not necessarily the case for $r=5$.  One natural  construction is to keep the sets $X,Y,Z$ from Definition~\ref{Def:KrHypergraph}. This fails, as $x_i,x_{i+1}, y_i, z_i, z_{i+20}$ spans an induced $K_5^-$.
It seems that to avoid this, we need to put  $5$-edges on $x_i, x_{i+1}$ that intersects $Y$ in different vertices. Our intuition suggested that we should avoid triangles in the $2$-skeleton, coming from different $5$-edges. This led us to use the Behrend construction, which is useful constructing such graphs. 
 To guarantee this property for $r=5$, we will build the hypergraph from sets of integers that are $3$-AP-free.

First, given a subset $B$ of integers, we define an auxiliary $5$-uniform hypergraph $\mathcal H_B(n)$  on $\Theta(n)$ vertices with $\Theta (|B|n)$ edges, which is induced $K_5^-$-free for an appropriate choice of $B$. 
The hypergraph $\mathcal H_B(n)$ however will not satisfy  condition (ii) of our Key Lemma. By cutting $\mathcal H_B(n)$ into shorter pieces and connecting them through certain ``turning gadgets'', 
we will define a new $5$-uniform hypergraph $\mathcal H'_B(n)$ (also on $\Theta(n)$ vertices with $\Theta (|B|n)$ edges),  
which now possesses condition (ii) but also preserves the induced $K_5^-$-free property.
 	
We construct our hypergraphs from very simple building blocks.
\begin{definition}[Chain]\label{chai2} A chain $\mathcal C$ of length $m$ is a $5$-uniform hypergraph on an ordered set $\{ w_1, \ldots , w_{3m+2} \}$ of vertices, with edge set $$\mathcal C = \{ e_i =\{w_{3i-2}, w_{3i-1}, w_{3i}, w_{3i+1}, w_{3i+2} \} : i \in [m] \}.$$
\end{definition}
Chains satisfy condition (i) and a much stronger condition (ii) of the Key Lemma.
\begin{lemma}[Key Lemma for chains]\label{lemma:KeyChain} 
Let $\mathcal C = \{ e_1, \ldots , e_m \}$ be a chain of length $m$. Then
\begin{itemize} 
\item[(i)] $\mathcal C$ is induced $K_5^-$-free.
\item[(ii)] For every $i\in [m-1]$, $|e_i\cap e_{i+1}|=2$, and $|e_i\cap e_j|=0$ 
for every $j, |i-j|>1$.
\end{itemize} 
\end{lemma}
\begin{proof} Part (ii) is immediate from the definition. %{\bf I dont get this sentence:} Part (ii) follows from (ii). 
For (i) let $j_1 < j_2 < j_3 < j_4 < j_5$ be the indices of a copy $K$ 
of $K_5^-$ in the $2$-skeleton $G$ of ${\mathcal C}$. If $w_{j_1}w_{j_5}$ is an edge of $G$ then 
$j_5-j_1 = 4$, and $j_3$ is of the form $3i$, hence $K$ is spanned by the hyperedge $e_i$. 
Otherwise either $w_{j_1}$ or $w_{j_5}$ is missing two edges into $V(K)$. 
%For $i,j$ with $|i-j|>2$, no vertex of $e_i$ is adjacent to a vertex of $e_j$ in the $2$-skeleton of ${\mathcal C}$. Thus it is sufficient to check the claim for a chain of length $3$, for which the reader is invited to do by hand. 
\end{proof}

\begin{definition}[$\mathcal H_b(n)$]\label{def:H_b(n)}
Let $n\in \mathbb{N}$ and $b\in [n]$. Let $X=\{x_0, \dots, x_{n}\}$, $Y=\{y_0, \dots, y_{n}\}$, and $Z=\{z_0, \dots, z_{n}\}$ be three pairwise disjoint sets.
The chain $\mathcal H_b(n)$ is defined on the vertex set $X\cup Y\cup Z$ with vertex order $x_0, z_{2b}, y_{b}, \ldots , x_i, z_{i+2b}, y_{i+b}, \ldots , y_{n-b}, x_{n-2b}, z_n$. 
\end{definition}

\paragraph{{\bf Remark.}} We have $$\mathcal H_b(n) = \{ E_{i,b}=\{x_{i}, x_{i+1}, y_{i+b}, z_{i+2b}, z_{i+2b+1}\} : \mbox{$0\leq i\leq n -2b-1$} \}$$ and for $i=0, \dots, n-2b-2$, we have $E_{i,b}\cap E_{i+1, b}=\{x_{i+1}, z_{i+2b+1}\}$ and for $j\geq i+2$, we have $E_{i,b}\cap E_{j, b}=\emptyset$.

Chains  have only  linearly many edges, we construct our first hypergraph by taking the union of several of them.

\begin{definition}[$\mathcal H_B(n)$]
For a subset $B\subseteq [n]$, we define $\mathcal H_B(n) = \cup_{b\in B} {\mathcal H_b(n)}$.
\end{definition}

Next we show that for an appropriately chosen $B$ the hypergraph  $\mathcal H_B(n)$ satisfies condition (i) of the Key Lemma (Lemma~\ref{Lem:KeyLemma}).

\begin{lemma}[$\mathcal H_B(n)$ is induced $K^-_5$-free]\label{Lemma_GB_is_K5_closed}
If $B=10B'$ for some $3$-AP-free set $B'$, then 
 $\mathcal H_B(n)$ is induced $K^-_5$-free.
\end{lemma}

\begin{proof} We show the lemma through a couple of claims.
We denote by $G_b$ and $G_B$ the $2$-skeleton of $\mathcal H_b(n)$ and $\mathcal H_B(n)$, respectively.
We say that an edge is a {\bf transverse edge} if it has the form $x_iy_j$, $y_iz_j$, or $x_iz_j$ for some $i,j$. 
We define the {\bf length} of transverse edges by $\Length(x_iy_j)=|j-i|$, $\Length $ $(z_k  y_i)=|k-j|$, $\Length(x_iz_k)=|k-i|/2$.

\begin{claim}[Length of transverse edges]\label{claim:length} \ 
Let $b,b'\in B$, $b\neq b'$ be two distinct integers. If $e \in G_b$ and 
$e'\in G_{b'}$ are transverse edges then $|\Length(e) -\Length(e')| \geq 8$.\\ 
In particular
every transverse edge $e\in G_B$ has a unique $b_e \in B$ such that $e \in G_{b_e}$.
\end{claim}
\begin{proof} 
Observe that the length of each of the eight transverse edges 
contained in a hyperedge $E_{i,b}=\{x_{i}, x_{i+1}, y_{i+b}, z_{i+2b}, z_{i+2b+1}\}$ is between $b- 1$ and $b+1$. Therefore, since the distance between any two distinct elements of the set $B=10B'$ is at least $10$, we have 
$$|\Length(e)-\Length(e')| \geq |b-b'| - |\Length(e) -b| - |\Length(e') -b'| 
\geq  10 -1 -1 \geq 8.$$ 
Taking $e=e'$ shows the uniqueness of the $b\in B$ for which $G_b$ contains $e$.
\end{proof}

A triangle is called a {\bf  transverse triangle} if all its edges are transverse (or equivalently if its vertices are $x_i, y_j, z_k$ for some $i,j,k$).

\begin{claim}[Triangles in $G_B$]\label{claim:triangles}
 In $G_B$ every triangle $T$ has at least two transverse edges and is contained in $G_b$ for some $b$. 
\end{claim}
\begin{proof}
The first statement holds since $G_B[X]$ is a sub-path of the path $(x_0, x_1, \ldots , x_n)$, $G_B[Z]$ is a sub-path of the path $(z_0, z_1, \ldots , z_n)$, and $G_B[Y]$ is an independent set.

For the second statement suppose first that $x_i$, $y_j$, $z_k$ are the vertices of a transverse triangle $T$ in $G_B$. This means that 
there exist $b_{xy},  b_{yz}, b_{xz} \in B$, such that $x_iy_j \in G_{b_{xy}}$, 
$y_jz_k \in G_{b_{yz}}$, and $x_iz_k\in G_{b_{xz}}$. By the description of the edges 
in Definition~\ref{def:H_b(n)} we then have $|j-i-b_{xy}|\leq 1$, $|k-j-b_{yz}|\leq 1$, $|k-i-2b_{xz}|\leq 1$. 
Using the triangle inequality we have  $|2b_{xz}-b_{xy}-b_{yz}|\leq |j-i-b_{xy}|+|k-j-b_{yz}|+|2b_{xz}+i-k|\leq 3$. Dividing this inequality by $10$, we obtain that for the elements $b'_{xy}:=\frac{b_{xy}}{10}, b'_{yz}:=\frac{b_{yz}}{10}, b'_{xz}:=\frac{b_{xz}}{10}$ of the set $B'$, it holds that
$|2b'_{xz}-b'_{xy}-b'_{yz}|\leq \frac{3}{10}$. Since $2b'_{xz}-b'_{xy}-b'_{yz}$ is an integer and $B'$ is $3$-AP-free, we have $b'_{xz} = b'_{xy} = b'_{yz}$. Hence
$b_{xz}=b_{xy}=b_{yz} = :b$ and $T\subseteq G_b$.

Assume now that $T$ is  a non-transverse triangle. By the first line of the proof, $T$ has either 
two vertices from $X$ or two vertices from $Z$. These vertices are adjacent in $G_B$, so have their indices $1$ apart: they are either $x_i$ and $x_{i+1}$ or $z_i$ and $z_{i+1}$ for some $i$. In any case, for the length of the transverse edges $e$ and $e'$ from these two vertices to the third vertex of $T$, we have $|\Length(e) - \Length(e')| \leq 1$. 
By Claim~\ref{claim:length}, $e$ and $e'$ are both contained in $G_b$ for some $b$, which also implies that the non-transverse of edge of $T$ 
is also contained in the same $G_b$.
\end{proof}

We are now ready to complete the proof of Lemma~\ref{Lemma_GB_is_K5_closed}.
Let $K$ be a copy of $K_5^-$ in $G_B$. Let $v_1, v_2 \in V(K)$ be the vertices of the missing edge of $K$ and let $U = V(K) \setminus \{ v_1, v_2\}$ be the remaining set of vertices spanning a triangle $T$ in $K$. 
By Claim~\ref{claim:triangles}, we have $T \subseteq G_b$ for some $b\in B$ and there are at least two transverse edges in $T$. Every other transverse edge $\tilde{e}$ of $K$ is adjacent to some transverse edge $e$ of $T$ and is contained  
in a triangle $\tilde{T}$ of $K \subseteq G_B$ together with $e$.  Applying Claim~\ref{claim:triangles}  to $\tilde{T}$, we have $b_e = b_{\tilde{e}}$. 

 Therefore
there exists $b\in B$ such that $b_e=b$ for every transverse edge $e$ of $K$. Then also all the non-transverse edges are in $G_b$, since they each participate in triangles of $K\subseteq G_B$. By Claim~\ref{claim:triangles}, these triangles must also contain transverse edges and hence also be contained in the same $G_b$. Hence the entire $K$ is in $G_b$, which by Lemma~\ref{lemma:KeyChain}(i) means that it is contained in an edge of $\mathcal H_b(n) \subseteq \mathcal H_B(n)$.
\end{proof}

Our construction of $\mathcal H_B(n)$ has condition (i) of the Key Lemma, but it lacks condition (ii), in particular there is no clear order on its edges. Generally, we plan to patch the chains of 
$\mathcal H_B(n)$ one after the other, but in order to be able to do so and still preserve condition (i), we might have to prune the chains a bit and connect their ends via short and disjoint ``turning gadgets''. We describe the construction of our final hypergraph $\mathcal H'_B(n)$ in two steps.

\begin{definition}[$\mathcal H'_b(n)$]\label{Def:H'b}
	Let $B = \{b_1,\dots, b_t\} \subseteq [n]$ be a set with its elements in order and let $f_{b_i}^j=\{x_{j}, z_{j+2b_i}\}$. 
	Define $\mathcal H'_{b_1}(n)=\mathcal H_{b_1}(n)$ and set $s_1=0$, $\ell_1=n-2b_1$. 
For every $2\leq j\leq t$, let $s_j$ and $\ell_j\leq n-2b_j$ be chosen such that $\ell_j-s_j$ is maximal with respect to $\left( f^{s_j}_{b_j} \cup f^{\ell_j}_{b_j} \right) 
\cap \left(\bigcup_{i=1}^{j-1} \left( f^{s_i}_{b_i}\cup f^{\ell_i}_{b_i}\right)  \right)=\emptyset$, and define $ \mathcal H'_{b_j}(n)=\{E_{i,b_j} \mid s_j\leq i\leq \ell_j-1 \}$. 
\end{definition}

\begin{definition}[$\mathcal H'_B(n)$]\label{Def:H'}
	Let $B=\{b_1,\dots,b_t \}$ and $\mathcal H'_{b_j}(n)$ be as in Definition~\ref{Def:H'b} and let $U_1, \ldots , U_{t-1}$ be pairwise disjoint new sets with $7$ vertices each.  For every $i\in [t-1]$, let $\mathcal D_i$ be a chain of  length $3$ on vertex set $f_{b_i}^{\ell_i} \cup U_i \cup f_{b_{i+1}}^{s_{i+1}}$, 
starting on $f_{b_i}^{\ell_i}$ and ending on $f_{b_{i+1}}^{s_{i+1}}$.
We define $\mathcal H'_B(n)=(\bigcup_{j=1}^t \mathcal H'_{b_j}(n))\cup (\bigcup_{i=1}^{t-1} \mathcal D_i)$. 
\end{definition}

We start by counting the vertices and edges of $\mathcal H'_B(n)$, and verifying
condition (ii) of the Key Lemma for it. 

\begin{proposition}\label{prop:simple-propertiesK5}
If $B\subseteq [n/4]$ is of the form $B=10B'$ where $B'$ is $3$-AP-free, 
then the following hold. 
	\begin{itemize} 
		\item[(a)] $\mathcal H'_B(n)$ has at most $3n+7|B|\leq 10n$ vertices.
		\item[(b)]
		$\mathcal H'_B(n)$ has $m\geq n|B|/2-8|B|^2$ edges.
\item[(c)] There is an ordering $E_1, E_1, \ldots , E_m$ of the edges of $\mathcal H_B'(n)$ such that there exist subsets $f_i  \subseteq E_i$ of size $|f_i|=2$ for $1 \leq t \leq m$, such that $f_i \subseteq E_j$ if and only if $i=m=j$ or $i <m$ and $j = i, i+1$.
	\end{itemize} 
\end{proposition}

\begin{proof} Part (a) follows by adding up the sizes of participating pairwise disjoint sets. For part (b), first note that $|\mathcal H_{b_j}(n)| = n-2b_j$. To construct $\mathcal H_{b_j}'(n)$ according to Definition~\ref{Def:H'b} we might need to delete from the beginning and the end of the chain $\mathcal H_{b_j}(n)$ the hyperedges incident to $\bigcup_{i=1}^{j-1} (f^{s_i}_{b_i} \cup f^{\ell_i+1}_{b_i})$. Each of these $4(j-1)$ vertices  participates in at most two edges of $\mathcal H_{b_j}(n)$ and hence $|\mathcal H_{b_j}'(n)| \geq n-2b_j - 8(j-1) \geq \frac{n}{2} - 8|B|$, where we also used $B\subseteq [n/4]$. The promised lower bound on  
$|\mathcal H_B'(n)| \geq \sum_{b\in B} |\mathcal H_b'(n)|$  follows.

For part (c) we create the ordering of $\mathcal H'_B(n)$ by patching together the natural ordering of the participating chains in the following order:
$\mathcal H'_{b_1}(n), \mathcal D_1,{\mathcal H'_{b_2}(n)},\mathcal D_2,\dots, \mathcal D_{t-1},{\mathcal H'_{b_t}(n)}$. The order of the edges within each $\mathcal H'_{b_i}(n)$ starts with $E^{s_i}_{b_i}$ and ends at $E^{\ell_i-1}_{b_i}$. Then come the edges of $\mathcal D_i$, starting with the edge containing $f^{\ell_i}_{b_i}$ and ending with the one containing $f^{s_i}_{b_{i+1}}$, after which $\mathcal H_{b_{i+1}}'(n)$ follows. In this ordering the intersections $E_i \cap E_{i+1} = : f_i$ have exactly two elements since they are either within a participating chain or one of the pairs $f_{b_i}^{s_i}$ or $f_{b_i}^{\ell_i}$ connecting two of these chains which are disjoint otherwise.
Furthermore $f_i$ is not contained in any other $E_j$, $j\neq i, i+1$ since in $U_i$ 
each vertex is in at most two hyperedges, and otherwise, by Claim~\ref{claim:length} each pair $\{x_i,z_{i+2b}\}$ appears in a hyperedge of some chain $\mathcal H_b(n)$ for a unique $b\in B$ (and is disjoint from each $U_i$).
\end{proof}

Next we also verify condition (i) of the Key Lemma for 
$\mathcal H'_B(n)$. %, that is, we show that $\mathcal H'_B(n)$ is induced $K_5^-$-free. 
\begin{proposition}[$\mathcal H'_B(n)$ is induced $K_5^-$-free]\label{prop:H'(B)IsClosed}
Suppose that $B=10B'$ for $B'$ $3$-AP-free. Then $\mathcal H'_B(n)$ is induced $K^-_5$-free.
\end{proposition}

\begin{proof}
Let $S_i \subseteq U_{i-1}$ and $L_i\subseteq U_i$ be the three-element sets, such that $f_{b_i}^{s_i} \cup S_i$ and $f_{b_i}^{\ell_i} \cup L_i$ form a hyperedge of $\mathcal D := \cup_{i=1}^{t-1} \mathcal D_i$. Then, by our construction, edges of the $2$-skeleton of $\mathcal H_B'(n)$ between $U: = \cup_{i=1}^{t-1} U_i$ and $V:= X\cup Y \cup Z$ only go between the sets $f_{b_i}^{s_i}$ and $S_i$ and 
between the sets $f_{b_i}^{\ell_i}$ and $L_i$. In particular, every vertex of $U$ has degree at most two into $V$.

Let $K$ be a copy of $K_5^-$ with vertex set $W \cup I$, where 
$W= U \cap V(K)$ and $I = V \cap V(K)$. 
We classify the cases according to $|W|$ and, unless
$|W|= 0, 5$, or $3$ and $W=S_i$ or $L_i$, arrive to a contradiction 
with the fact that in $K$ there is {\em one non-edge}.

If $W$ is empty then $V(K)$ is fully contained in $V$,
on which the $2$-skeleton of $\mathcal D$ only induces the edges $f_{b_i}^{s_i}$ and $f_{b_i}^{\ell_i}$, and hence the $2$-skeleton of 
$\mathcal H_B'(n)$ on $V$ is the same as the $2$-skeleton of $\cup_{b\in B} \mathcal H_b'(n)$, which is induced $K_5^-$-free as the subhypergraph of the induced $K_5^-$-free hypergraph  $\mathcal H_B(n)$ (see  Lemma~\ref{Lemma_GB_is_K5_closed}).

If $W$ consists of the single vertex $v$, then $v$ has at most two edges into the set $V$, so at least two non-edges to the set $I$. 

If $W$ consists of two vertices, then both of them have at least one non-edge towards the three vertices of $I$, contradicting that $K$ has at most one non-edge.

If $W$ consists of three vertices then, unless it is equal to a single $S_i$ or $L_i$, 
each of the two vertices in $I$ will have at least one non-edge towards $W$, contradicting that $K$ has only one non-edge. If $W = S_i$ (or $W=L_i$), then both vertices of $I$ must be in $f_{b_i}^{s_i}$ (or in $f_{b_i}^{\ell_i}$, respectively),  otherwise there are at least three non-edges in $K$. In this case $V(K)$ is $f_{b_i}^{s_i}\cup S_i$ (or $f_{b_i}^{\ell_i}\cup L_i$) which is a hyperedge of $\mathcal D$.

If $W$ contains four vertices and it contains some $S_i$ or  $L_i$, then 
it induces at least one non-edge of $K$ and the vertex $v \in I$ also has one non-edge towards $W$, which is at least two non-edges in $K$, a contradiction. Otherwise $W$ does not contain any $S_i$ or $L_i$ and then $v$ has at least two non-edges towards $W$.

Finally if $|W| =5$, then $K$ is part of $\mathcal D$, which is induced $K_5^-$-free as it is the disjoint union of chains.
\end{proof}

\begin{proof}[Proof of Theorem~\ref{thm:mainK5}]
Let $B'\subseteq [n/400]$ be a $3$-AP-free subset of size $r_3(n/400)$. Then 
for the set $B=10B' \subset [n/40]$ the hypergraph $\mathcal H_B'(n/10)$  satisfies condition (ii) of the Key Lemma by Proposition~\ref{prop:simple-propertiesK5}(c) and Condition (i) by Proposition~\ref{prop:H'(B)IsClosed}. Consequently, by parts (a) and (b) of Proposition~\ref{prop:simple-propertiesK5}, for large enough $n$
\begin{align*}
 M_5(n)&\geq M_5(|V({\mathcal H'_B(n/10)})|) \geq |{\mathcal H'_B(n/10)}|\geq n|B|/20-8|B|^2\\ 
&\geq n|B|/30\geq nr_3(n)/1200,
\end{align*}

since $r_3(n/400) \geq r_3(n)/400$ by the Pigeonhole Principle.
\end{proof}

\bigskip

\section*{Acknowledgements} 
This work was initiated during the  Oberwolfach Mini-Workshop on Positional Games in October 2018. 
The authors would like to thank the participants of the workshop, and in particular to Asaf Ferber, for helpful discussions.

\end{document}